\documentclass[runningheads,a4paper]{llncs}
\usepackage{amssymb}
\usepackage{graphicx}
\usepackage{url}

\urldef{\mailsa}\path|devsi.bantva@gmail.com|
\newcommand{\keywords}[1]{\par\addvspace\baselineskip
\noindent\keywordname\enspace\ignorespaces#1}

\begin{document}
\mainmatter  

\title{On hamiltonian colorings of trees}

\author{Devsi Bantva}
\institute{Lukhdhirji Engineering College, Morvi - 363 642 \\
Gujarat (INDIA) \\
\mailsa\\}

\toctitle{On hamiltonian colorings of trees}
\tocauthor{Devsi Bantva}
\maketitle

\begin{abstract}
A \emph{hamiltonian coloring} $c$ of a graph $G$ of order $n$ is a mapping $c$ : $V(G) \rightarrow \{0,1,2,...\}$ such that $D(u, v)$ + $|c(u) - c(v)|$ $\geq$ $n-1$, for every two distinct vertices $u$ and $v$ of $G$, where $D(u, v)$ denotes the detour distance between $u$ and $v$ which is the length of a longest $u,v$-path in $G$. The value \emph{hc(c)} of a hamiltonian coloring $c$ is the maximum color assigned to a vertex of $G$. The hamiltonian chromatic number, denoted by $hc(G)$, is the min\{$hc(c)$\} taken over all hamiltonian coloring $c$ of $G$. In this paper, we present a lower bound for the hamiltonian chromatic number of trees and give a sufficient condition to achieve this lower bound. Using this condition we determine the hamiltonian chromatic number of symmetric trees, firecracker trees and a special class of caterpillars.
\end{abstract}

\section{Introduction}
\qquad A \emph{hamiltonian coloring} $c$ of a graph $G$ of order $n$ is a mapping $c$ : $V(G) \rightarrow \{0,1,2,...\}$ such that $D(u, v)$ + $|c(u) - c(v)|$ $\geq$ $n-1$, for every two distinct vertices $u$ and $v$ of $G$, where $D(u, v)$ denotes the detour distance between $u$ and $v$ which is the length of a longest $u$,$v$-path in $G$. The value of $hc(c)$ of a hamiltonian coloring $c$ is the maximum color assigned to a vertex of $G$. The \emph{hamiltonian chromatic number} $hc(G)$ of $G$ is min\{$hc(c)$\} taken over all hamiltonian coloring $c$ of $G$. It is clear from definition that two vertices $u$ and $v$ can be assigned the same color only if $G$ contains a hamiltonian $u$,$v$-path, and hence a graph $G$ can be colored by a single color if and only if $G$ is hamiltonian-connected. Thus the hamiltonian chromatic number of a connected graph $G$ measures how close $G$ is to being hamiltonian-connected. The concept of hamiltonian coloring was introduced by Chartrand \emph{et al.}\cite{Cha1} which is a variation of \emph{radio $k$-coloring} of graphs.

At present, the hamiltonian chromatic number is known only for handful of graph families. Chartrand \emph{et al.}\cite{Cha1,Cha2} determined the  hamiltonian chromatic number for complete graph $K_{n}$, cycle $C_{n}$, star $K_{1,k}$, complete bipartite graph $K_{r,s}$ and presented upper bound for the hamiltonian chromatic number of paths and trees. The exact value of hamiltonian chromatic number of paths which is equal to the radio antipodal number $ac(P_{n})$ was given by Khennoufa and Togni in \cite{Togni}. Shen \emph{et al.}\cite{Shen} discussed the hamiltonian chromatic number for graphs $G$ with max\{$D(u,v)$ : $u,v \in V(G)$, $u \neq v$\} $\leq$ $n/2$, where $n$ is the order of graph $G$; such graphs are called graphs with maximum distance bound $n/2$ or $DB\left(n/2\right)$ graphs for short and they determined the hamiltonian chromatic number for double stars and a special class of caterpillars.

In this paper, we present a lower bound for the hamiltonian chromatic number of trees (Theorem \ref{lower:thm}) and give a sufficient condition to achieve this lower bound (Theorem \ref{main:thm}). Using this condition we determine the hamiltonian chromatic number of symmetric trees, firecracker trees and a special class of caterpillars. We use an approach similar to the one used in \cite{Bantva} to derive a lower bound of the hamiltonian chromatic number of trees. We remark that our proof for the hamiltonian chromatic number of a special class of caterpillars is simple than one given in \cite{Shen} by different approach. We also inform the readers that the hamiltonian chromatic number obtain in this paper is one less than that defined in \cite{Cha1,Cha2,Cha3,Cha4,Shen} as we allowed 0 for coloring while they do not.

\section{Preliminaries}

\indent\indent A \emph{tree} is a connected graph that contains no cycle. The \emph{diameter} of $T$, denoted by $diam(T)$ or simply $d$, is the maximum distance among all pairs of vertices in $T$. The \emph{eccentricity} of a vertex in a graph is the maximum distance from it to other vertices in the graph, and the \emph{center} of a graph is the set of vertices with minimum eccentricity. It is well known that the center of a tree $T$, denoted by $C(T)$, consists of a single vertex or two adjacent vertices, called the \emph{central vertex/vertices} of $T$. We view $T$ as rooted at its central vertex/vertices; if $T$ has only one central vertex $w$ then $T$ is rooted at $w$ and if $T$ has two adjacent central vertices $w$ and $w^{'}$ then $T$ is rooted at $w$ and $w^{'}$ in the sense that both $w$ and $w^{'}$ are at level 0. If $u$ is on the path joining another vertex $v$ and central vertex $w$, then $u$ is called \emph{ancestor} of $v$, and $v$ is a \emph{descendent} of $u$. Let $u \not\in C(T)$ be adjacent to a central vertex. The subtree induced by $u$ and all its descendent is called a \emph{branch} at $u$. Two branches are called \emph{different} if they are at two vertices adjacent to the same central vertex, and \emph{opposite} if they are at two vertices adjacent to different central vertices. Define the \emph{detour level} of a vertex $u$ from the center of graph by
\begin{eqnarray*}
\mathcal{L}(u) := min\{D(u,w):w \in C(T)\}, u \in V(T).
\end{eqnarray*}
Define the \emph{total detour level} of $T$ as
\begin{eqnarray*}
\mathcal{L}(T) := \displaystyle\sum_{u \in V(T)}\mathcal{L}(u).
\end{eqnarray*}
For any $u, v \in V(T)$, define $\phi(u,v)$ := max\{$\mathcal{L}(t) : t$ is a common ancestor of $u$ and $v$\}, and
\begin{eqnarray*}
\delta(u,v) := \left\{
\begin{array}{ll}
1, & \mbox{if $C(T)$ = \{$w,w^{'}$\} and path $P_{uv}$ contains an edge $ww^{'}$}, \\ [0.3cm]
0, & \mbox{otherwise.}
\end{array}
\right.
\end{eqnarray*}
\begin{lemma} Let $T$ be a tree with diameter $d \geq 2$. Then for any $u, v \in V(T)$ the following holds:
\begin{enumerate}
  \item $\phi(u,v) \geq 0$;
  \item $\phi(u,v) = 0$ if and only if $u$ and $v$ are in different or opposite branches;
  \item $\delta(u,v)$ = 1 if and only if $T$ has two central vertices and $u$ and $v$ are in opposite branches;
  \item the detour distance $D(u,v)$ in $T$ between $u$ and $v$ can be expressed as
  \begin{equation} D(u,v) = \mathcal{L}(u) + \mathcal{L}(v) - 2 \phi(u,v) + \delta(u,v).\label{eq:dist}\end{equation}
\end{enumerate}
\end{lemma}
Note that for a tree $T$ the detour distance $D(u,v)$ is same as the ordinary distance $d(u,v)$ as there is unique path between any two vertices $u$ and $v$ of $T$. Thus, one can use expression (\ref{eq:dist}) for ordinary distance $d(u,v)$ which can also be used for other purpose.

Define
\begin{eqnarray*}
\varepsilon(T) := \left\{
\begin{array}{ll}
0, & \mbox{if $C(T) = \{w\}$}, \\ [0.3cm]
1, & \mbox{if $C(T) = \{w,w^{'}\}$}.
\end{array}
\right.
\end{eqnarray*}
\begin{eqnarray*}
\varepsilon^{'}(T) := 1 - \varepsilon(T).
\end{eqnarray*}

\section{On Hamiltonian colorings of trees}

\indent\indent For a connected graph $G$ of order $n \geq 5$, by defining $D(\sigma)$ = $\sum_{i=1}^{n-1}D(v_{i},v_{i+1})$ for an ordering $\sigma$ : $v_{1}$, $v_{2}$,$....$,$v_{n}$ and $D(G)$ = max\{$D(\sigma)$ : $\sigma$ is an ordering of $V(G)$\}, Chartrand \emph{et al.}\cite{Cha3} established the following lower bound for the hamiltonian chromatic number of a connected graph $G$.

\begin{theorem}(\cite{Cha3}) If $G$ is a connected graph of order $n \geq 5$, then $hc(G) \geq (n-1)^{2} + 1 - D(G)$.
\end{theorem}

For an ordering $\sigma$ : $v_{1}$, $v_{2}$,$....$,$v_{n}$ of the vertices of $G$, define $c_{\sigma}$ to be an assignment of positive integers to $V(G)$: $c_{\sigma}(v_{1})$ = 1 and $c_{\sigma}(v_{i+1})-c_{\sigma}(v_{i})$ = $(n-1)-D(v_{i},v_{i+1})$ for each $1 \leq i \leq n-1$. If max\{$D(u,v)$ : $u, v \in V(G), u \neq v$\} $\leq n/2$ for a connected graph $G$ of order $n$ then such a graph $G$ is called a graph with \emph{maximum distance bound $n/2$} or $DB(n/2)$ graph for short. Shen \emph{et al.}\cite{Shen} proved the following Theorems about $DB(n/2)$ graphs and using it determined the hamiltonian chromatic number for double stars and a special class of caterpillars.

\begin{theorem}(\cite{Shen}) Let $G$ be a $DB(n/2)$ graph of order $n \geq 4$. Then for any $\sigma$, $c_{\sigma}$ is a hamiltonian coloring for $G$ with hc$(c_{\sigma})$ = $(n-1)^{2} + 1 - D(\sigma)$.
\end{theorem}

\begin{theorem}(\cite{Shen}) If $G$ is $DB(n/2)$ graph of order $n \geq 5$, then $hc(G)$ = $(n-1)^{2} + 1 - D(G)$, and for any $\sigma$ such that $D(\sigma)$ = $D(G)$, hc($c_{\sigma}$) = hc($G$). Namely, $c_{\sigma}$ is a minimum hamiltonian coloring for $G$.
\end{theorem}

Now, let $T$ be a tree with maximum degree $\Delta$. Note that a hamiltonian coloring $c$ of $T$ is injective for $\Delta(T) \geq 3$ as in this case no two vertices of $T$ contain hamiltonian path. Throughout this section we consider $T$ with $\Delta(T) \geq 3$ then $c$ induces a linear order of the vertices of $T$, namely $V(T)$ = \{$u_{0},u_{1},...,u_{n-1}$\} (where $n$ = $|V(T)|$) such that
\begin{center}
0 = $c(u_{0}) < c(u_{1}) < ... < c(u_{n-1})$ = span($c$).
\end{center}

\begin{theorem} \label{lower:thm} Let $T$ be a tree of order $n \geq 4$ and $\Delta(T) \geq 3$. Then
\begin{equation} hc(T) \geq (n-1)(n-1-\varepsilon(T)) + \varepsilon^{'}(T) - 2\mathcal{L}(T). \label{eq:lower}\end{equation}.
\end{theorem}
\begin{proof} It is enough to prove that any hamiltonian coloring of $T$ has span not less than the right-hand side of (\ref{eq:lower}). Suppose $c$ is any hamiltonian coloring of $T$ then $c$ order the vertices of $T$ into a linear order $u_{0}$, $u_{1}$,...,$u_{n-1}$ such that 0 = $c(u_{0}) < c(u_{1}) < ... < c(u_{n-1})$. By definition of $c$, we have $c(u_{i+1}) - c(u_{i}) \geq n - 1 - D(u_{i}, u_{i+1})$ for $0 \leq i \leq n-1$. Summing up these $n-1$ inequalities, we obtain
\begin{equation} \mbox{span}(c) = c(u_{n-1}) \geq (n-1)^{2} - \displaystyle\sum_{i=0}^{n-1} D(u_{i},u_{i+1}) \label{eq:sumup}\end{equation}
\textsf{Case-1:} $T$ has one central vertex.~~In this case, we have $\phi(u_{i},u_{i+1}) \geq 0$ and $\delta(u_{i},u_{i+1})$ = 0 for $0 \leq i \leq n-2$ by the definition of the function $\phi$ and $\delta$. Since $T$ has only one central vertex, $u_{0}$ and $u_{n-1}$ cannot be the central vertex of $T$ simultaneously. Hence $\mathcal{L}(u_{0}) + \mathcal{L}(u_{n-1}) \geq 1$. Thus, by substituting (\ref{eq:dist}) in (\ref{eq:sumup}),
\begin{eqnarray*}
\mbox{span}(c) &\geq& (n-1)^{2} - \displaystyle\sum_{i=0}^{n-1} \left[\mathcal{L}(u_{i}) + \mathcal{L}(u_{i+1}) - 2 \phi(u_{i},u_{i+1}) + \delta(u_{i},u_{i+1})\right] \\
&=& (n-1)^{2} - 2\displaystyle\sum_{i=0}^{n-1} \mathcal{L}(u_{i}) + \mathcal{L}(u_{0}) + \mathcal{L}(u_{n-1}) - 2 \displaystyle\sum_{i=0}^{n-1} \phi(u_{i},u_{i+1}) \\
&\geq& (n-1)^{2} + 1 - 2 \mathcal{L}(T) \\
&=& (n-1)(n-1-\varepsilon(T)) + \varepsilon^{'}(T) - 2 \mathcal{L}(T).
\end{eqnarray*}
\noindent\textsf{Case-2:} $T$ has two central vertices.~~In this case, we have $\phi(u_{i},u_{i+1}) \geq 0$ and $\delta(u_{i},u_{i+1}) \leq 1$ for $0 \leq i \leq n-2$ by the definition of the function $\phi$ and $\delta$. Since $T$ has two central vertices, we can set $\{u_{0},u_{n-1}\}$ = $\{w,w^{'}\}$. Thus, by substituting (\ref{eq:dist}) in (\ref{eq:sumup}),
\begin{eqnarray*}
\mbox{span}(c) &\geq& (n-1)^{2} - \displaystyle\sum_{i=0}^{n-1} \left[\mathcal{L}(u_{i}) + \mathcal{L}(u_{i+1}) - 2 \phi(u_{i},u_{i+1}) + \delta(u_{i},u_{i+1})\right] \\
&=& (n-1)^{2} - 2 \displaystyle\sum_{i=0}^{n-1} [\mathcal{L}(u_{i}) + \mathcal{L}(u_{i+1})] - 2 \displaystyle\sum_{i=0}^{n-1} \phi(u_{i},u_{i+1}) + \displaystyle\sum_{i=0}^{n-1} \delta(u_{i},u_{i+1}) \\
&=& (n-1)^{2} - 2 \displaystyle\sum_{i=0}^{n-1} \mathcal{L}(u_{i}) + \mathcal{L}(u_{0}) + \mathcal{L}(u_{n-1}) + \displaystyle\sum_{i=0}^{n-1} \delta(u_{i},u_{i+1}) \\
&\geq& (n-1)^{2} - 2 \displaystyle\sum_{u \in V(T)} \mathcal{L}(u_{i}) + (n-1) \\
&=& (n-1)(n-2) - 2\mathcal{L}(T) \\
&=& (n-1)(n-1-\varepsilon(T)) + \varepsilon^{'}(T) - 2\mathcal{L}(T).
\end{eqnarray*}
\end{proof}

\begin{theorem} \label{main:thm} Let $T$ be a tree of order $n \geq 4$ and $\Delta(T) \geq 3$. Then
\begin{equation} hc(T) = (n-1)(n-1-\varepsilon(T)) + \varepsilon^{'}(T) - 2\mathcal{L}(T) \label{eq:main}\end{equation} holds if there exists a linear order $u_{0}$, $u_{1}$,...,$u_{n-1}$ with 0 = $c(u_{0}) < c(u_{1}) < ... < c(u_{n-1})$ of the vertices of $T$ such that
\begin{enumerate}
   \item $u_{0}$ = $w$, $u_{n-1} \in N(w)$ when $C(T)$ = $\{w\}$ and $\{u_{0},u_{n-1}\}$ = $\{w,w^{'}\}$ when $C(T)$ = $\{w,w^{'}\}$,
   \item $u_{i}$ and $u_{i+1}$ are in different branches when $C(T)$ = $\{w\}$ and opposite branches when $C(T)$ = $\{w,w^{'}\}$,
   \item $D(u_{i},u_{i+1}) \leq n/2$, for $0 \leq i \leq n-2$.
 \end{enumerate}
Moreover, under these conditions the mapping $c$ defined by
\begin{equation} c(u_{0}) = 0 \label{eq:f0} \end{equation}
\begin{equation} c(u_{i+1}) = c(u_{i}) + n - 1 - \mathcal{L}(u_{i}) - \mathcal{L}(u_{i+1}) - \varepsilon(T), 0 \leq i \leq n-2 \label{eq:f1} \end{equation}
is an optimal hamiltonian coloring of $T$.
\end{theorem}
\begin{proof} Suppose that a linear order $u_{0}, u_{1},...,u_{n-1}$ of the vertices of $T$ satisfies the conditions (1), (2) and (3) of hypothesis, and $c$ is defined by (\ref{eq:f0}) and (\ref{eq:f1}). By Theorem \ref{lower:thm}, it is enough to prove that $c$ is a hamiltonian coloring whose span is equal to $c(u_{n-1})$ = $(n-1)(n-1-\varepsilon(T)) + \varepsilon^{'}(T) - 2\mathcal{L}(T)$.

Let $c$ is defined by (\ref{eq:f0}) and (\ref{eq:f1}). Without loss of generality we assume that $j-i \geq 2$. Then
\begin{eqnarray*}
c(u_{j}) - c(u_{i}) &=& \displaystyle\sum_{t = i}^{j-1} [c(u_{t+1}) - c(u_{t})] \\
&=& \displaystyle\sum_{t = i}^{j-1} [n - 1 - \mathcal{L}(u_{t}) - \mathcal{L}(u_{t+1}) - \varepsilon(T)] \\
&=& \displaystyle\sum_{t = i}^{j-1} [n - 1 - D(u_{t}, u_{t+1})] \\
&=& (j - i)(n - 1) - \displaystyle\sum_{t = i}^{j-1} D(u_{t}, u_{t+1}) \\
&\geq& (j - i)(n - 1) - (j - i)\left(\frac{n}{2}\right) \\
&=& (j - i)\left(\frac{n - 2}{2}\right) \\
&\geq& n - 2
\end{eqnarray*}

\noindent Note that $D(u_{i}, u_{j}) \geq 1$; it follows that $|c(u_{j}) - c(u_{i})| + D(u_{i}, u_{j}) \geq n - 1$. Hence, $c$ is a hamiltonian coloring for $T$. The span of $c$ is given by
\begin{eqnarray*}
\mbox{span}(c) &=& c(u_{n-1}) - c(u_{0}) \\
&=& \displaystyle\sum_{t = 0}^{n-2} [c(u_{t+1}) - c(u_{t})] \\
&=&\displaystyle\sum_{t = 0}^{n-2} [n - 1 - \mathcal{L}(u_{t}) - \mathcal{L}(u_{t+1}) - \varepsilon(T)] \\
&=& (n-1)^{2} - \displaystyle\sum_{t = 0}^{n-2} [\mathcal{L}(u_{t}) + \mathcal{L}(u_{t+1})] - (n-1)\varepsilon(T) \\
&=& (n-1)(n-1-\varepsilon(T)) - 2\mathcal{L}(T) + \mathcal{L}(u_{0}) + \mathcal{L}(u_{n-1}) \\
&=& (n-1)(n-1-\varepsilon(T)) + \varepsilon^{'}(T) - 2\mathcal{L}(T)
\end{eqnarray*}
Therefore, $hc(T) \leq (n-1)(n-1-\varepsilon(T)) + \varepsilon^{'}(T) - 2\mathcal{L}(T)$. This together with (\ref{eq:lower}) implies (\ref{eq:main}) and that $c$ is an optimal hamiltonian coloring.
\end{proof}

\begin{corollary} \label{dbp2:thm}Let $T$ be a $DB(n/2)$ tree (or $d \leq n/2$) of order $n \geq 4$ and $\Delta(T) \geq 3$, where $d$ is diameter of $T$. Then
\begin{equation}\label{eq:dbp2} hc(T) = (n-1)(n-1-\varepsilon(T))+\varepsilon^{'}(T)-2\mathcal{L}(T) \end{equation}
holds if there exists a linear order $u_{0}$, $u_{1}$,...,$u_{n-1}$ with 0 = $c(u_{0}) < c(u_{1}) < ... < c(u_{n-1})$ of the vertices of $T$ such that
\begin{enumerate}
   \item $u_{0}$ = $w$, $u_{n-1} \in N(w)$ when $C(T)$ = $\{w\}$ and $\{u_{0},u_{n-1}\}$ = $\{w,w^{'}\}$ when $C(T)$ = $\{w,w^{'}\}$,
   \item $u_{i}$ and $u_{i+1}$ are in different branches when $C(T)$ = $\{w\}$ and opposite branches when $C(T)$ = $\{w,w^{'}\}$.
 \end{enumerate}
Moreover, under these conditions the mapping $c$ defined by
\begin{equation} c(u_{0}) = 0 \label{eq:f00}\end{equation}
\begin{equation} c(u_{i+1}) = c(u_{i}) + n - 1 - \mathcal{L}(u_{i}) - \mathcal{L}(u_{i+1}) - \varepsilon(T), 0 \leq i \leq n-2 \label{eq:f11}\end{equation}
is an optimal hamiltonian coloring of $T$.
\end{corollary}
\begin{proof} The proof is straight forward by Theorem \ref{main:thm} as for any tree $T$, max\{$D(u,v) : u, v \in V(G), u \neq v$\} $\leq$ $d$ $\leq$ $n/2$.
\end{proof}

\section{Hamiltonian coloring of some families of tree}

\indent\indent In this section, we determine the hamiltonian chromatic number for three families of tree using Corollary \ref{dbp2:thm}. We continue to use terminology and notation defined in the previous section.

\indent A \emph{symmetric tree} is a tree in which all vertices other than leaves (degree-one vertices) have the same degree and all leaves have the same eccentricity. Let $k, d \geq 2$ be integers. We denote the symmetric tree with diameter $d$ and non-leaf vertices having degree $k+1$ by $T_{k+1}(d)$. A $k$-star is a tree consisting of $k$ leaves and another vertex joined to all leaves by edges. We define the $(n,k)$-\emph{firecracker trees}, denoted by $F(n,k)$, to be the tree obtained by taking $n$ copies of a $(k-1)$-star and identifying a leaf of each of them to a different vertex of a path of length $n-1$. A tree is said to be a \emph{caterpillar} $C$ if it consists of a path $v_{1}v_{2}...v_{m}(m \geq 3)$, called the spine of $C$, with some hanging edges known as legs, which are incident to the inner vertices $v_{2}$,$v_{3}$,...,$v_{m-1}$. If $d(v_{i})$ = $k$ for $i = 2,3,...,m-1$, then we denote the caterpillar by $C(m,k)$, where $d(v_{i})$ denotes the degree of $v_{i}$. For all above defined trees it is easy to verify that $d \leq n/2$, and hence $DB(n/2)$ trees as max\{$D(u,v)$:$u,v \in V(T)$\} $\leq$ $d$ $\leq$ $n/2$.

\indent Now we determine the hamiltonian chromatic number for above defined trees using Corollary \ref{dbp2:thm}. Note that for this purpose it is enough to give a linear order $u_{0},u_{1},...,u_{n-1}$ of vertices of $T$ which satisfies conditions of Corollary \ref{dbp2:thm}.

\begin{theorem} Let $k, d \geq 2$ be integers. Then $hc(T_{k+1}(d))$
\begin{equation} \label{sym:hc}
= \left\{
\begin{array}{ll}
\frac{(k+1)^{2}}{(k-1)^{2}}(k^{\frac{d}{2}}-1)\left[(k^{\frac{d}{2}}-1)+\frac{1}{k+1}(2-(k-1)d)\right]-\frac{k+1}{k-1}d+1, & \mbox{if $d$ is even}, \\ [0.3cm]
\frac{4k}{(k-1)^{2}}(k^{\frac{d-1}{2}}-1)\left[k(k^{\frac{d-1}{2}}-1)+1\right]+\frac{2k}{k-1}(2-d)k^{\frac{d-1}{2}}-\frac{2k}{k-1}, & \mbox{if $d$ is odd.}
\end{array}
\right.
\end{equation}
\end{theorem}
\begin{proof} Note that $T_{k+1}(d)$ has one or two central vertex/vertices depending on $d$ and hence we consider the following two cases.

\indent\textsf{Case 1}: $d$ is even.

In this case $T_{k+1}(d)$ has a unique central vertex, denoted by $w$. Denote the children of the central vertex $w$ by $w^{1}, w^{2}, \ldots, w^{k+1}$. Denote the $k$ children of each $w^{t}$ by $w_{0}^{t}, w_{1}^{t}, \ldots, w_{k-1}^{t}$, $1 \le t \le k+1$. Denote the $k$ children of each $w_{i}^{t}$ by $w_{i0}^{t}, w_{i1}^{t}, \ldots, w_{i(k-1)}^{t}$, $0 \leq i \leq k-1$, $1 \leq t \leq k+1$. Inductively, denote the $k$ children of $w_{i_{1},i_{2},\ldots,i_{l}}^{t}$ ($0 \leq i_{1}, i_{2}, \ldots, i_{l} \leq k-1$, $1 \leq t \leq k+1$) by $w_{i_{1}, i_{2},\ldots,i_{l}, i_{l+1}}^{t}$ where $0 \le i_{l+1} \le k-1$. Continue this until all vertices of $T_{k+1}(d)$ are indexed this way. We then rename the vertices of $T_{k+1}(d)$ as follows:

For $1 \leq t \leq k+1$, set
$$
v_{j}^{t} := w_{i_{1},i_{2}, \ldots, i_{l}}^{t},\;\, \mbox{where}\;\, j = 1 + i_{1} + i_{2}k + \cdots + i_{l}k^{l-1} + \sum_{l+1 \leq t \leq \lfloor d/2 \rfloor} k^{t}.
$$

We give a linear order $u_{0}, u_{1}, \ldots, u_{n-1}$ of the vertices of $T_{k+1}(d)$ as follows. We first set $u_{0} = w$. Next, for $1 \leq j \leq n-k-2$, let
$$
u_{j} := \left\{
\begin{array}{ll}
v_{s}^{t},\;\,\mbox{where $s = \lceil j/(k+1) \rceil$}, & \mbox{if $j \equiv t$ (mod $(k+1)$) for $t$ with $1 \le t \le k$}, \\ [0.3cm]
v_{s}^{k+1},\;\,\mbox{where $s = \lceil j/(k+1) \rceil$}, & \mbox{if $j \equiv 0$ (mod $(k+1)$)}.
\end{array}
\right.
$$
Finally, let
$$
u_{j} := w^{j-n+k+2},\;\, n-k-1 \leq j \leq n-1.
$$
Note that $u_{n-1} = w^{k+1}$ is adjacent to $w$, and for $1 \leq i \leq n-2$, $u_{i}$ and $u_{i+1}$ are in different branches so that $\phi(u_{i},u_{i+1}) = 0$.

\textsf{Case 2}: $d$ is odd.

In this case $T_{k+1}(d)$ has two (adjacent) central vertices, denoted by $w$ and $w'$. Denote the neighbours of $w$ other than $w'$ by $w_{0}, w_{1}, \ldots, w_{k-1}$ and the neighbours of $w'$ other than $w$ by $w'_{0}, w'_{1}, \ldots, w'_{k-1}$. For $0 \le i \le k-1$, denote the $k$ children of each $w_{i}$ (respectively, $w'_{i}$) by $w_{i0}, w_{i1}, \ldots, w_{i(k-1)}$ (respectively, $w'_{i0}, w'_{i1}, \ldots, w'_{i(k-1)}$). Inductively, for $0 \leq i_{1},i_{2},\ldots,i_{l} \leq k-1$, denote the $k$ children of $w_{i_{1},i_{2}, \ldots, i_{l}}$ (respectively, $w'_{i_{1},i_{2},\ldots,i_{l}}$) by $w_{i_{1},i_{2},\ldots,i_{l},i_{l+1}}$ (respectively, $w'_{i_{1},i_{2},\ldots,i_{l},i_{l+1}}$), where $0 \leq i_{l+1} \leq k-1$.
We rename
$$
v_{j} := w_{i_{1},i_{2},\ldots,i_{l}},\;\, v'_{j} := w'_{i_{1},i_{2},\ldots,i_{l}},\;\,\mbox{where}\;\, j = 1 + i_{1} + i_{2}k + \cdots + i_{l}k^{l-1} + \sum_{l+1 \leq t \leq \lfloor d/2 \rfloor} k^{t}.
$$
We give a linear order $u_{0}, u_{1}, \ldots, u_{n-1}$ of the vertices of $T_{k+1}(d)$ as follows. We first set
$$
u_0 := w,\;\, u_{n-1} := w',
$$
and for $1 \leq j \leq n-2$, let
\begin{eqnarray*}
u_j := \left\{
\begin{array}{ll}
v_{s},\;\,\mbox{where $s = \lceil j/2 \rceil$}, & \mbox{if $j \equiv 0$ (mod $2$)}\\ [0.3cm]
v'_{s},\;\,\mbox{where $s = \lceil j/2 \rceil$}, & \mbox{if $j \equiv 1$ (mod $2$)}.
\end{array}
\right.
\end{eqnarray*}
Then $u_{i}$ and $u_{i+1}$ are in opposite branches for $1 \leq i \leq n-2$, and $u_{i+2j}$, $j$=0,1,...,$(k-1)$  are in different branches for $1 \leq i \leq n-2k+1$, so that $\phi(u_{i},u_{i+1}) = 0$ and $\delta(u_{i},u_{i+1}) = 1$.

Therefore, in each case above, a defined linear order of vertices satisfies the conditions of Corollary \ref{dbp2:thm}. The hamiltonian coloring defined by (\ref{eq:f00}) and (\ref{eq:f11}) is an optimal hamiltonian coloring whose span equal to the right-hand side of (\ref{eq:dbp2}). But it is straight forward to verify that the order of $T_{k+1}(d)$ is given by
\begin{equation}
\label{sym:n}
n := \left\{
\begin{array}{ll}
1+\frac{k+1}{k-1} (k^{\frac{d}{2}} - 1), & \mbox{if $d$ is even}, \\ [0.3cm]
2 \left(1 + \frac{k}{k-1} (k^{\frac{d-1}{2}} - 1)\right), & \mbox{if $d$ is odd}.
\end{array}
\right.
\end{equation}
With the help of formula $1 + 2x + 3x^{2} + ... + px^{p-1} = \frac{px^{p}}{x-1} - \frac{x^{p}-1}{(x-1)^{2}}$, one can verify that the total level of $T_{k+1}(d)$ is given by
\begin{equation}
\label{sym:l}
\mathcal{L}(T_{k+1}(d)) := \left\{
\begin{array}{ll}
(k+1)\left(\frac{d k^{\frac{d}{2}}}{2(k-1)} - \frac{k^{\frac{d}{2}}-1}{(k-1)^{2}}\right), & \mbox{if $d$ is even} \\ [0.3cm]
2k\left(\frac{(d-1) k^{\frac{d-1}{2}}}{2(k-1)} - \frac{k^{\frac{d-1}{2}}-1}{(k-1)^{2}}\right), & \mbox{if $d$ is odd}.
\end{array}
\right.
\end{equation}
By substituting (\ref{sym:n}) and (\ref{sym:l}) into (\ref{eq:dbp2}), we obtain the right-hand side of (\ref{sym:hc}) is the hamiltonian chromatic number of $T_{k+1}(d)$. \\
\end{proof}

\begin{theorem} For $m \geq 3$ and $k \geq 4$,
\begin{equation} \label{fire:hc}
hc(F(m,k)) = \left\{
\begin{array}{ll}
m^{2}k^{2}-6m(k-1)-\frac{k}{2}(m^{2}-1)+2, & \mbox{if $m$ is odd}, \\ [0.3cm]
m^{2}k^{2}-6m(k-1)-\frac{k}{2}m^{2}+2, & \mbox{if $m$ is even.}
\end{array}
\right.
\end{equation}
\end{theorem}
\begin{proof} Let $w^{i}_{1}$,$w^{i}_{2}$,...,$w^{i}_{k}$ denote the vertices of the $i^{th}$ copy of the $(k-1)$-star in $F(m,k)$, where $w^{i}_{1}$ is the apex vertex (center) and $w^{i}_{2}$,...,$w^{i}_{k}$ are the leaves. Without loss of generality we assume that $w^{1}_{k}$, $w^{2}_{k}$,...,$w^{m}_{k}$ are identified to the vertices in the path of length $m-1$ in the definition of $F(m,k)$. Note that $F(m,k)$ has one or two central vertex/vertices depending on $m$ and hence we consider the following two cases.

\textsf{Case-1:} $m$ is odd.

In this case $F(m,k)$ has only one central vertex $w$ which is $w^{\lfloor\frac{m}{2}\rfloor}_{k}$. We give a linear order $u_{0}$, $u_{1}$,...,$u_{n-1}$ of the vertices of $F(m,k)$ as follows. We first set $u_{0}$ = $w$ = $w_{k}^{\lfloor\frac{m}{2}\rfloor}$. Next, for $1 \leq t \leq n-m$, let
\begin{eqnarray*}
u_{t} := w^{i}_{j}, \mbox{ where  } t = \left\{
\begin{array}{ll}
(j-1)m + (i-\lfloor\frac{m}{2}\rfloor), & \mbox{if } i = \lfloor\frac{m}{2}\rfloor \\ [0.3cm]
(j-1)m + 2i, & \mbox{if } i < \lfloor\frac{m}{2}\rfloor \\ [0.3cm]
(j-1)m + 2(i-\lfloor\frac{m}{2}\rfloor)+1, & \mbox{if } i > \lfloor\frac{m}{2}\rfloor.
\end{array}
\right.
\end{eqnarray*}

Finally, for $n-m+1 \leq t \leq n-1$, let
\begin{eqnarray*}
u_{t} := w^{i}_{j}, \mbox{ where  } t = \left\{
\begin{array}{ll}
(j-1)m - 2(i-\lfloor\frac{m}{2}\rfloor)+1, & \mbox{if } i < \lfloor\frac{m}{2}\rfloor \\ [0.3cm]
(j-1)m + 2(m-i+1), & \mbox{if } i > \lfloor\frac{m}{2}\rfloor.
\end{array}
\right.
\end{eqnarray*}

\textsf{Case-2:} $m$ is even.

In this case $F(m,k)$ has two central vertices $w$ and $w^{'}$ which are $w_{k}^{\frac{m}{2}}$ and $w_{k}^{\frac{m}{2}+1}$ respectively. We give a linear order $u_{0}$, $u_{1}$,...,$u_{n-1}$ of the vertices of $F(m,k)$ as follows. We first set $u_{0}$ = $w^{'}$ = $w_{k}^{\frac{m}{2}+1}$ and $u_{n-1}$ = $w$ = $w_{k}^{\frac{m}{2}}$. Next, for $1 \leq t \leq n-m+1$, let
\begin{eqnarray*}
u_{t} := w^{i}_{j}, \mbox{ where  } t = \left\{
\begin{array}{ll}
(j-1)m + 2i - 1, & \mbox{if } i \leq \frac{m}{2} \\ [0.3cm]
(j-1)m + 2(i-\frac{m}{2}), & \mbox{if } i > \frac{m}{2}.
\end{array}
\right.
\end{eqnarray*}

Finally, for $n-m+2 \leq t \leq n-2$, let
\begin{eqnarray*}
u_{t} := w^{i}_{j}, \mbox{ where  } t = \left\{
\begin{array}{ll}
(j-1)m + 2i - 1, & \mbox{if } i < \frac{m}{2} \\ [0.3cm]
(j-1)m + 2(i-1-\frac{m}{2}), & \mbox{if } i > \frac{m}{2}+1.
\end{array}
\right.
\end{eqnarray*}

Therefore, in each case above, a defined linear order of vertices satisfies conditions of Corollary \ref{dbp2:thm}. The hamiltonian coloring defined by (\ref{eq:f00}) and (\ref{eq:f11}) is an optimal hamiltonian coloring whose span equal to the right-hand side of (\ref{eq:dbp2}). But the order and total level of firecrackers $F(m,k)$ are given by
\begin{equation} n := mk \label{fire:p} \end{equation}
\begin{equation} \label{fire:l}
\mathcal{L}(F(m,k)) := \left\{
\begin{array}{ll}
\frac{km^{2}+(8k-12)m-k}{4}, & \mbox{if $m$ is odd}, \\ [0.3cm]
\frac{km^{2}+6m(k-2)}{4}, & \mbox{if $m$ is even.}
\end{array}
\right.
\end{equation}
By substituting (\ref{fire:p}) and (\ref{fire:l}) into (\ref{eq:dbp2}), we obtain the right-hand side of (\ref{fire:hc}) is the hamiltonian chromatic number of $F(m,k)$.
\end{proof}

\begin{theorem} Let $m, k \geq 3$. Then $hc(C(m,k))$
\begin{equation} \label{cater:hc}
= \left\{
\begin{array}{ll}
(m-2)^{2}k^{2}-\frac{1}{2}(5m^{2}-20m+19)k+\frac{1}{2}(3m^{2}-12m+11), & \mbox{if $m$ is odd}, \\ [0.3cm]
(m-2)^{2}k^{2}-\frac{1}{2}(5m^{2}-20m+20)k+\frac{1}{2}(3m^{2}-12m+12), & \mbox{if $m$ is even.}
\end{array}
\right.
\end{equation}
\end{theorem}
\begin{proof} Let $v_{1}$, $v_{2}$,...,$v_{m}$ be the vertices of spine and $v_{i}^{j}$, $1 \leq j \leq k-2$ are pendent vertices at $i^{th}$, $2 \leq i \leq m-1$ vertex of spine. Note that $C(m,k)$ has one or two central vertex/vertices depending on $m$ and hence we consider the following two cases.

\textsf{Case-1:} $m$ is odd.

In this case $C(m,k)$ has only one central vertex which is $v_{\lfloor\frac{m}{2}\rfloor}$ = $w$. We first set $u_{0}$ = $v_{\lfloor\frac{m}{2}\rfloor+1}$, $u_{n-1}$ = $w$ and other vertices as follows.

For $1 \leq t \leq m-2$,
\begin{eqnarray*}
u_{t} := v_{i}, \mbox{ where } t = \left\{
\begin{array}{ll}
2i-1, & \mbox{if $i < \lfloor\frac{m}{2}\rfloor$}, \\ [0.3cm]
2(i-\lfloor\frac{m}{2}\rfloor), & \mbox{if $i > \lfloor\frac{m}{2}\rfloor+1$.}
\end{array}
\right.
\end{eqnarray*}

For $m-1 \leq t \leq n-1$,
\begin{eqnarray*}
u_{t} := v_{i}^{j}, \mbox{ where } t = \left\{
\begin{array}{ll}
(m-2)j+2(i-1), & \mbox{if $i < \lfloor\frac{m}{2}\rfloor$}, \\ [0.3cm]
(m-2)j+1, & \mbox{if $i = \lfloor\frac{m}{2}\rfloor$}, \\ [0.3cm]
(m-2)j+2(i-\lfloor\frac{m}{2}\rfloor)+1, & \mbox{if $i > \lfloor\frac{m}{2}\rfloor$.}
\end{array}
\right.
\end{eqnarray*}

\textsf{Case-2:} $m$ is even.

In this case $C(m,k)$ has two central vertices which are $v_{\frac{m}{2}}$ = $w$ and $v_{\frac{m}{2}+1}$ = $w^{'}$. We first set $u_{0}$ = $v_{\frac{m}{2}+1}$, $u_{n-1}$ = $\frac{m}{2}$ and other vertices as follows.

For $1 \leq t \leq m-2$,
\begin{eqnarray*}
u_{t} := v_{i}, \mbox{ where } t = \left\{
\begin{array}{ll}
2i-1, & \mbox{if $i < \frac{m}{2}-1$}, \\ [0.3cm]
2(i-\frac{m}{2}), & \mbox{if $i > \frac{m}{2}+1$.}
\end{array}
\right.
\end{eqnarray*}

For $m-1 \leq t \leq n-1$,
\begin{eqnarray*}
u_{t} := v_{i}^{j}, \mbox{ where } t = \left\{
\begin{array}{ll}
(m-2)j+2(i-2)+1, & \mbox{if $i \leq \frac{m}{2}$}, \\ [0.3cm]
(m-2)j+2(i-\frac{m}{2}), & \mbox{if $i > \frac{m}{2}$.}
\end{array}
\right.
\end{eqnarray*}

Therefore, in each case above, a defined linear order of vertices satisfies conditions of Corollary \ref{dbp2:thm}. The hamiltonian coloring defined by (\ref{eq:f00}) and (\ref{eq:f11}) is an optimal hamiltonian coloring whose span equal to the right-hand side of (\ref{eq:dbp2}). But the order and total level of caterpillars $C(m,k)$ are given by
\begin{equation} \label{cater:p} n := m(k-1)-2(k-2) \end{equation}
\begin{equation} \label{cater:l}
\mathcal{L}(C(m,k)) := \left\{
\begin{array}{ll}
\frac{(m^{2}-5)(k-1)}{4}+1, & \mbox{if $m$ is odd}, \\ [0.3cm]
\frac{m(m-2)(k-1)}{4}, & \mbox{if $m$ is even.}
\end{array}
\right.
\end{equation}
By substituting (\ref{cater:p}) and (\ref{cater:l}) into (\ref{eq:dbp2}), we obtain the right-hand side of (\ref{cater:hc}) is the hamiltonian chromatic number of $C(m,k)$.
\end{proof}

We remark that Theorem \ref{main:thm} is also useful to determine hamiltonian chromatic number of non $DB(n/2)$ trees. See the following result.
\begin{theorem} Let $P_{m}^{'}$ be a tree obtained by attaching a pendant vertex to central vertex/vertices of path $P_{m}$. Then
\begin{equation}\label{hc:Pn}
hc(P_{m}^{'}) := \left\{
\begin{array}{ll}
\frac{1}{2}(m^{2}-1), & \mbox{if $m$ is odd}, \\ [0.3cm]
\frac{m^{2}}{2}+2m-4, & \mbox{if $m$ is even}.
\end{array}
\right.
\end{equation}
\end{theorem}
\begin{proof} The order and total level of $P_{m}^{'}$ are given by
\begin{equation}\label{n:Pn}
n := \left\{
\begin{array}{ll}
m+1, & \mbox{if $m$ is odd}, \\ [0.3cm]
m+2, & \mbox{if $m$ is even}.
\end{array}
\right.
\end{equation}
\begin{equation}\label{L:Pn}
\mathcal{L}(P_{m}^{'}) := \left\{
\begin{array}{ll}
\frac{m^{2}+3}{4}, & \mbox{if $m$ is odd}, \\ [0.3cm]
\frac{m^{2}-2m+8}{4}, & \mbox{if $m$ is even}.
\end{array}
\right.
\end{equation}
Substituting (\ref{n:Pn}) and (\ref{L:Pn}) into (\ref{eq:lower}) we obtain that the right-hand side of (\ref{hc:Pn}) is a lower bound for $hc(P_{m}^{'})$. Now we give a linear ordering of vertices of $P_{m}^{'}$ which satisfies conditions of Theorem \ref{main:thm}. Note that $P_{m}^{'}$ has one central vertex when $m$ is odd and two adjacent central vertices when $m$ is even. Hence we consider the following two cases.

\textsf{Case-1:} $m$ is odd.

Let $v_{1}v_{2}...v_{m}$ be the vertices of path and $v^{'}$ be the vertex attached to central vertex $v_{(m+1)/2}$ then we order the vertices as follows:
\begin{eqnarray*}
v_{(m+1)/2}, v_{1}, v_{(m+3)/2}, v_{2}, v_{(m+5)/2}, v_{3}, v_{(m+7)/2},....,v_{(m-1)/2}, v_{m}, v^{'}.
\end{eqnarray*}
Rename the vertices of $P_{m}^{'}$ in the above ordering by $u_{0}$, $u_{1}$,...,$u_{n-1}$. Namely, let $u_{0}$ = $v_{(m+1)/2}$, $u_{1}$ = $v_{1}$,...,$u_{n-1}$ = $v^{'}$ then it satisfies conditions of Theorem \ref{main:thm}.

\textsf{Case-2:} $m$ is even.

Let $v_{1}v_{2}...v_{m}$ be the vertices of path and $v^{'}$ and $v^{''}$ are attached to central vertices $v_{m/2}$ and $v_{m/2+1}$ then we order the vertices as follows:
\begin{eqnarray*}
v_{m/2+1}, v_{1}, v_{m/2+2}, v_{2}, v_{m/2+3}, v_{3},....,v_{m/2-1}, v_{m}, v^{'}, v^{''}, v_{m/2}.
\end{eqnarray*}
Rename the vertices of $P_{m}^{'}$ in the above ordering by $u_{0}$, $u_{1}$,...,$u_{n-1}$. Namely, let $u_{0}$ = $v_{m/2+1}$, $u_{1}$ = $v_{1}$,...,$u_{n-1}$ = $v_{m/2}$ then it satisfies conditions of Theorem \ref{main:thm}.

Therefore, in each case above, a defined linear order of vertices of $P_{m}^{'}$ satisfies conditions of Theorem \ref{main:thm} and hence the hamiltonian coloring defined by (\ref{eq:f0}) and (\ref{eq:f1}) is an optimal hamiltonian coloring whose span is (\ref{eq:main}) which is (\ref{hc:Pn}) for the current case.
\end{proof}

\section*{Acknowledgement}
I want to express my deep gratitude to an anonymous referee for kind comments and constructive suggestions.

\end{document}